\documentclass[11pt, a4paper]{amsart}
\usepackage{amsmath, amssymb, amsthm}
\usepackage{color}
\usepackage{tikz}
\usepackage{tikz-cd}
\usepackage{mathrsfs}

\newif\ifdebug

\debugtrue

\def\sK{\mathcal{K}}
\newcommand{\SK}{\mathcal{K}}
\def\zk{\mathcal{Z}_{\mathcal{K}}}
\newcommand{\momang}{\mathcal{Z}_{\mathcal{K}}}

\newcommand{\Hom}{\mathrm{Hom}}

\def\st{\mathop\mathrm{st}}

\addtocounter{section}{-1}

\def\Ker{\mathop{\mathrm{Ker}}}

\def\F{\mathcal{F}}

\def\Max{\mathscr C}
\def\C{\mathbb C}
\def\R{\mathbb R}
\def\Z{\mathbb Z}
\def\g{\mathfrak g}
\def\h{\mathfrak h}
\def\r{\mathfrak r}

\def\ge{\geqslant}

\def\le{\leqslant}

\theoremstyle{plain}
\newtheorem{thm}{Theorem}[section]
\newtheorem{lemma}[thm]{Lemma}

\newtheorem{proposition}[thm]{Proposition}

\theoremstyle{definition}
\newtheorem{defn}[thm]{Definition} 

\newtheorem{rmk}[thm]{Remark}
\newtheorem{constr}[thm]{Construction}

\begin{document}
\title[Dolbeault cohomology of manifolds with torus action]{Dolbeault cohomology of complex manifolds with torus action}

\author{Roman Krutowski}
\address{Faculty of Mathematics, National Research University Higher School of Economics, Moscow, Russia}
\email{roman.krutovskiy@protonmail.com}

\author{Taras Panov}
\address{Department of Mathematics and Mechanics, Lomonosov Moscow
State University;
Faculty of Computer Science, National Research University Higher School of
Economics, Moscow, Russia; and
Institute for Information Transmission Problems, Russian Academy of Sciences, Moscow}
\email{tpanov@mech.math.msu.su}
\thanks{The first author was supported by the Laboratory of Mirror
Symmetry NRU HSE, by the Simons Foundation  and by the Foundation for the Advancement of Theoretical Physics and
Mathematics ``BASIS''}
\thanks{The research of the second author was carried out within the framework of the Basic Research Program at HSE University and funded by the Russian Academic Excellence Project ``5-100''.}

\begin{abstract}
We describe the basic Dolbeault cohomology algebra of the canonical foliation on a class of complex manifolds with a torus symmetry group. This class includes complex moment-angle manifolds, LVM- and LVMB-manifolds and, in most generality, complex manifolds
with a maximal holomorphic torus action. We also provide a dga model for the ordinary Dolbeault cohomology algebra. The Hodge decomposition for the basic Dolbeault cohomology is proved by reducing to the transversely K\"ahler (equivalently, polytopal) case using a foliated analogue of toric blow-up.
\end{abstract}


\maketitle

\setcounter{section}{0}
\section{Introduction}
\emph{Complex moment-angle manifolds} $\zk$ are non-K\"ahler complex-analytic manifolds with a holomorphic action of the complex algebraic torus $(\C^\times)^m$~\cite{panov2012complex,p-u-v16}. The action of the compact tours $T^m\subset (\C^\times)^m$ on $\zk$ is \emph{maximal} in the sense that there exists $x\in\zk$ such that $m+\dim T^m_x=\dim\zk$, where $T^m_x$ is the stabiliser subgroup at~$x$. By a classification result of Ishida~\cite{ishida2013complex}, any complex manifold $M$ with a maximal torus action by holomorphic transformations can be obtained as the quotient $\zk/C$ by an appropriate freely acting closed subgroup $C\subset T^m$.

Complex moment-angle manifolds are a subclass of \emph{LVMB-manifolds}~\cite{meer00,bosi01}. The relationship between LVM-~\!, LVMB-manifolds and moment-angle-manifolds is described in~\cite{buchstaber2015toric,ba-za15,ishida2013complex}. In short, the underlying smooth manifold of an LVMB-manifold is either a moment-angle manifold $\zk$, or the quotient $\zk/S^1$ of a moment-angle manifold by the diagonal circle action. LVM-manifolds correspond to polytopal moment-angle manifolds~$\zk$.

A complex structure on a moment-angle manifold $\zk$ is defined in terms of a complete simplicial fan~$\Sigma$ with underlying simplicial complex~$\sK$. The manifold $\zk$ is equipped with a canonical holomorphic foliation $\mathcal F_{\mathfrak h}$ by the orbits of an action of a generally non-compact group, see the details in~\S\ref{subsecfol}. When the fan $\Sigma$ is rational, the foliation $\mathcal F_{\mathfrak h}$ becomes a holomorphic fibre bundle over the toric variety $V_\Sigma$ with fibres compact complex tori.

There is an important particular case when $\Sigma$ is the normal fan of a convex polytope. The corresponding moment-angle manifold $\zk$ is called \emph{polytopal}. In combinatorial terms, the underlying simplicial complex $\sK$ is a \emph{starshaped} sphere triangulation. A polytopal moment-angle manifold $\zk$ can be written as a nondegenerate intersection of Hermitian quadrics~\cite[Chapter~6]{buchstaber2015toric}. If $\Sigma$ is a nonsingular rational fan, then the corresponding toric variety $V_\Sigma$ is K\"ahler (which is equivalent to being  projective or symplectic) if and only if $\Sigma$ is polytopal. There is a foliated version of this fact which applies in the general situation: the foliation $\mathcal F_{\mathfrak h}$ is transversely K\"ahler if $\zk$ is polytopal~\cite[Proposition~4.4]{p-u-v16}. The converse is also true~\cite[Theorem~5.5]{ishida2018towards}, although a transversely K\"ahler form may exist on an open dense subset of $\zk$ under a milder condition on the fan~$\Sigma$, see~\cite[Theorem~4.6]{p-u-v16}. 

In the polytopal case, the transversely K\"ahler foliations $(\zk,\mathcal F_{\mathfrak h})$ can be viewed as irrational (or ``non-commutative'') analogues of symplectic toric manifolds and are studied by several groups of authors, including Battaglia and Prato~\cite{ba-pr01,ba-pr19}, Katzarkov, Lupercio, Meersseman and Verjovsky~\cite{k-l-m-v14}, Ratiu and Zung~\cite{ra-zu}.  In particular, several irrational versions of the Delzant correspondence between symplectic toric manifolds and their moment polytopes were obtained in these works.

Battaglia and Zaffran~\cite{ba-za15} considered basic cohomology of the canonical foliation $\mathcal F_{\mathfrak h}$ (in the context of LVMB-manifolds), computed the basic Betti numbers in the case when the associated fan $\Sigma$ is shellable, and proved that the basic cohomology ring $H^{*}_{\F_{\h}}(\momang)$ is generated by the classes of degree~2 when $\Sigma$ is polytopal. They conjectured that the basic cohomology ring has a description similar to the cohomology ring of a complete simplicial toric variety~\cite[\S5~(iv)]{ba-za15}, and that the basic Hodge numbers of $\F_{\h}$ are concentrated on the diagonal~\cite[\S5~(v)]{ba-za15}. The first conjecture is proved in~\cite{ikp2018basic}, while the question about the Hodge numbers is addressed here.

Since the foliation $\mathcal F_{\mathfrak h}$ is holomorphic, a natural question arises of whether its basic cohomology admits a Hodge decomposition. We prove this fact in Theorem~\ref{mainhodge} here, and also show that non-trivial basic Dolbeault cohomology groups $H^{p,q}_{\F_{\h}}(\momang)$ appear only on the diagonal $p=q$ of the Hodge diamond, as in the case of Dolbeault cohomology of a complete nonsingular toric variety. This gives a positive answer to the question of Battaglia and Zaffran mentioned above.  The basic Dolbeault cohomology algebra $H^{*,*}_{\F_{\h}}(\momang)$ is therefore fully described, see Theorem~\ref{DolbeaultM}.

To establish the Hodge decomposition for the basic cohomology of 
$(\zk,\mathcal F_{\mathfrak h})$ we introduce the notion of a \emph{Fujiki foliation}, a foliated version of a Fujiki class $\mathcal C$ manifold, see Section~\ref{fujikifol}. Namely, given a foliated moment-angle manifold $(\zk,\mathcal F_{\mathfrak h})$, we construct a holomorphic foliated surjection $\mathcal Z_{\mathcal K'} \to \zk$ from a transversely K\"{ahler} foliated moment-angle manifold~$\mathcal Z_{\mathcal K'}$. This is done by considering stellar subdivions of the original fan $\Sigma$ and defining an analogue of toric blow-up at the level of moment-agle manifolds, see~\S\ref{subm}. This construction allows us to reduce a general torus-invariant foliation to a trasverse K\"ahler (or transverse symplectic) one, and may be useful for other problems related to holomorphic foliations.

In the transverse K\"{a}hler (or polytopal) case the Dolbeault cohomology ring $H^{*,*}_{\F_{\h}}(\momang)$ was described in~\cite{ishida2018towards}. Using the representation of $\momang$ as an intersection of quadrics, Ishida showed in~\cite[Theorem~8.1]{ishida2018towards} that all generators of the Dolbeault cohomology ring are of type~$(1,1)$. Recently, Lin and Yang proved a more general result~\cite[Theorem~5.2]{li-ya} describing the Dolbeault cohomology of transverse K\"{a}hler foliations admiting a Hamiltonian torus action. 

In Section~\ref{max} we extend our description of the basic Dolbeault cohomology ring to general complex manifolds with holomorphic maximal torus action (which include LVM- and LVMB-manifolds). The main result here is Theorem~\ref{dolbgene}; its proof uses the notion of
\emph{transverse equivalence} and follows the approach developed in~\cite[\S5]{ikp2018basic}.

In the last section we use the basic Dolbeault cohomology ring to obtain a DGA model for the ordinary Dolbeault cohomology of a complex moment-angle manifold (Theorem~\ref{dgamodel}). This model extends the one obtained in~\cite[Theorem 5.4]{panov2012complex} in the case of a rational fan.

The authors thank Hiroaki Ishida for his valuable comments.

\section{Preliminaries: holomorphic foliations on complex manifolds.}

Here review basic facts about Riemannian and holomorphic foliations, following Molino~\cite{Molino} and  El Kacimi--Alaoui~\cite{kacimi1986decomposition}.

Let $\F$ be a holomorphic \emph{Hermitian} foliation of complex codimension $q$ on a compact complex manifold~$M$. We assume $\F$ to be \emph{homologically orientable}, that is, the top-degree basic cohomology group is nonzero: $H^{2q}_{\F}(M) \neq 0$. 
We denote by $T\F$ the complex vector bundle of tangent spaces to the leaves of~$\F$. We fix a transverse Hermitian metric, that is, a Hermitrian metric on the quotient complex vector bundle~$TM/T\F$. Using the associated transverse Riemannian metric we define the associated principal $SO(2q)$-bundle of transverse orthogonal oriented frames 
\[
  p_T\colon E_T(M, \F) \rightarrow M.
\]  
The foliation $\F$ lifts to a foliation $\F_T$ on $E_T(M, \F)$, see~\cite[Proposition~2.4]{Molino}. The foliation $\F_T$ is {\it transversely parallelisable}. That is, there are transverse vertical real vector fields $X_1,\ldots,X_{q(2q-1)}$ and transverse horizontal (with respect to the associated transverse Levi--Civita connection) vector fields $Y_1,\ldots,Y_{2q}$  which together form a basis of $T_zE_T(M, \F)/T_z\F_T$ at each point $z \in E_T(M, \F)$, see \cite[Section 5.1]{Molino}. 
By~\cite[Theorem 4.2]{Molino}, the closures of the leaves of the transversely parallelisable foliation $\F_T$ are the fibres of a locally trivial fibre bundle \[
  \pi_{{}_{Tb}}\colon E_T(M, \F) \rightarrow W_T,
\]  
called the \emph{basic fibration} over the \emph{basic manifold}~$W_T$. Set $s=\dim_\R W_T$. 

Let $\xi_1,\ldots,\xi_{q(2q-1)}$ be the  $1$-forms on $E_T(M, \F)$ satisfying $\xi_i(X_j)=\delta_{ij}$ and $\xi_i(Y_k)=0$ for $1 \le i, j \le q(2q-1)$ and $1 \le k \le 2q$. We put $\chi = \xi_1 \wedge \cdots \wedge \xi_{q(2q-1)}$.  Let $\Omega^*_{\F}(M)$ denote the algebra of real- or complex-valued basic forms with respect to the foliation~$\F$. We define a map 
\[
  S\colon \Omega^{2q}_{\F}(M) \rightarrow \Omega^{s}(W_T)
\]
as follows. Let $\lambda\in\Omega^*_{\F_T} (E_T(M, \F))$ be a form which restricts to a basic volume form on each fiber of $\pi_{{}_{Tb}}$ (see \cite[Proposition 3.2]{kacimi1986decomposition}). Given $\omega \in \Omega^{2q}_{\F}(M)$, we define $S(\omega) \in \Omega^{s}(W_T)$ as the form satisfying
\[
  p_T^*(\omega) \wedge \chi= \pi_{{}_{Tb}}^*S(\omega)\wedge\lambda.
\]

We say that a real basic form $\omega \in \Omega^{2q}_\F(M; \mathbb{R})$ is {\it positive} ({\it non-negative}) if $\omega(\rho)>0$ ($\omega(\rho) \ge 0$)  for any oriented $2q$-frame $\rho$ in $T_xM/T_x\F$. 
Note that if $\omega \in \Omega^{2q}_{\F}(M; \mathbb{R})$ is a positive (non-negative) basic form, then $S(\omega) \in \Omega^{s}(W_T ; \mathbb{R})$ is a positive (non-negative) form.

For a basic form $\omega \in \Omega^{2q}_{\F}(M, \mathbb{R})$, the {\it foliated integral} is defined by
\[
    \int_{M/\F}\omega := \int_{W_T}S(\omega).
\]

The following lemma is implicit in~\cite[Section~3]{kacimi1986decomposition}:
\begin{lemma}\label{positive}
Let $\omega \in \Omega^{2q}_{\F}(M, \mathbb{R})$ be a non-negative basic form which is positive on an open subset $U \subset M$. Then $\int_{M/\F}\omega > 0$. 
\end{lemma}
\begin{proof}
Since $\pi_{{}_{Tb}}$ is a fibre bundle, $\pi_{{}_{Tb}}\big(E_T(U, \F|_U)\big)\subset W_T$ is an open subset. The form $S(\omega)$ is non-negative and its restriction to $\pi_{{}_{Tb}}\big(E_T(U, \F|_U)\big)$ is positive. It follows that
\[
  \int_{M/\F}\omega \ge\int_{\pi_{{}_{Tb}}\big(E_T(U, \F|_U)\big)}S(\omega)>0.\qedhere
\]
\end{proof}

\section{Fujiki foliations}\label{fujikifol}
We introduce here a class of holomorphic foliations on complex manifolds, which can be regarded as a foliated version of \emph{Fujiki class~$\mathcal C$ manifolds} (see \cite[Lemma 4.6]{fujiki1978closedness}). We prove that the basic Dolbeault cohomology ring of a Fujiki foliation admits a Hodge decomposition. 

\begin{defn}
 We refer to a homologically orientable Hermitian foliation $(M, \mathcal{F})$ as a {\it Fujiki foliation} if there exists a homologically orientable transversely K\"{a}hler foliation $(M', \mathcal{F}')$ and a surjective holomorphic foliated map
\begin{equation*}
    f\colon (M', \mathcal{F}') \rightarrow (M, \mathcal{F}).
\end{equation*}
\end{defn}

By the result of El Kacimi--Alaoui~\cite[Theorem 3.4.6]{kacimi1990operateurs}, the basic Dolbeault cohomology ring $H^{*,*}_{\mathcal{F}}(M';\mathbb{C})$ of a transversely K\"{a}hler foliation admits a Hodge decomposition. Here we extend this result to Fujiki foliations.

\begin{lemma}\label{Fujiki}
Let $f$ be as above. Then the induced map in basic Dolbeault cohomology $f^*\colon H^{*,*}_{\mathcal{F}}(M;\mathbb{C}) \rightarrow H^{*,*}_{\mathcal{F'}}(M';\mathbb{C})$ is injective.
\end{lemma}

\begin{proof}
The proof below follows the lines of the standard argument in the non-foliated case, see e.\,g.~\cite[Lemma~7.28]{voisin76hodge}.

Let $s$ be the complex codimension of~$\F$, and let $t$ be the complex codimension of the restriction of $\F'$ to a generic fibre of~$f$, so that the codimension of $\F'$ is $s+t$. Denote by $\omega$ the transversely K\"{a}hler form on~$M'$.
 
As $f$ is holomorphic and surjective, there is an open subset $U' \subset M'$ such that $f|_{U'}$ is a trivial fibre bundle. More precisely, we may choose an open subset $U \subset M$ satisfying $f(U')=U$, an open chart $W\subset\C^k$ of a generic fibre of~$f$, and a fibrewise biholomorphism $U'\cong U \times W$.
 
Let $\gamma \in H^{2s}_{\F}(M; \mathbb{R})$ be a nonzero top-dimensional basic cohomology class, represented by a positive basic form $\sigma \in \Omega^{2s}_{\mathcal{F}}(M)$. Then $f^*(\sigma) \wedge \omega^t$ is a basic $2(s+t)$-form on $M'$. As in~\cite[Proposition 4.9]{kacimi1986decomposition}, in order to show that the cohomology class $f^*(\gamma) \wedge [\omega]^t\in H^{2(s+t)}_{\F'}(M'; \mathbb{R})$ is nonzero it is enough to prove the following inequality:
\begin{equation}\label{posint1}
\int_{M'/\F'}f^*(\sigma) \wedge \omega^t > 0.
\end{equation}


Since $\omega$ is a transversely K\"{a}hler form on $(M', \F')$, its restriction to $(\{x\}\times W,\F'|_{\{x\}\times W})$ is also transversely K\"{a}hler for each $x\in U$, because the embedding $(\{x\}\times W,\F'|_{\{x\}\times W})\to(M', \F')$ is holomorphic foliated. Therefore, $\omega^t$ is a positive basic form on 
$(\{x\} \times W , \F'_{\{x\}\times W})$. Furthermore, the restriction of $f^*(\sigma)$ to $U \times\{y\} \subset U \times W$ is a positive basic form for any $y\in W$. It follows that $f^*(\sigma) \wedge \omega^t$ is a nonnegative basic form on $(M', \F')$ which is positive when restricted to~$U'$. Now~\eqref{posint1} follows by applying Lemma~\ref{positive}.

As a consequence, we obtain that the top-degree Dolbeault cohomology map 
\[
f^*\colon H^{s,s}_{\F}(M; \mathbb{C})\to H^{s,s}_{\F'}(M'; \mathbb{C})
\]
induced by $f\colon M'\to M$ is injective, because $H^{s,s}_{\F}(M; \mathbb{C})=H^{2s}_{\F}(M; \mathbb{C})$. The rest follows by considering Serre duality for basic cohomology, see \cite[Theorem 3.3.4]{kacimi1990operateurs}. Indeed, take a nonzero $\alpha \in H^{p,q}_{\F}(M; \mathbb{C})$ and its Serre dual $\beta \in H^{s-p,s-q}_{\F}(M; \mathbb{C})$. Then $ \alpha \wedge \beta \in H^{s,s}_{\F}(M; \mathbb{C})$ is nonzero. Hence, $f^*(\alpha \wedge \beta)\neq 0$ in  $H^{s,s}_{\F'}(M'; \mathbb{C})$, so  $f^*(\alpha)\neq 0$ in $H^{p,q}_{\F'}(M'; \mathbb{C})$.
\end{proof}

As an immediate consequence we obtain

\begin{thm}\label{hodge}
For any Fujiki foliation $(M, \F)$ there is a Hodge decomposition
\begin{equation*}
    H^r_{\F}(M; \mathbb{C})= \bigoplus_{p+q=r}H^{p,q}_{\F}(M; \mathbb{C}).
\end{equation*}
\end{thm}

\begin{rmk}
Fujiki~\cite{fujiki1978closedness} introduced several equivalent conditions specifying his class $\mathcal C$ manifolds. In particular, a compact complex manifold $M$ belongs to the class $\mathcal C$ if one of the following is satisfied:
\begin{itemize}
\item[(a)] $M$ is a holomorphic image of a compact K\"ahler manifold;
\item[(b)] $M$ is a meromorphic image of a compact K\"ahler manifold;
\item[(c)] $M$ is a bimeromorphic to a compact K\"ahler manifold.
\end{itemize}
The equivalence of (a) and (b) is established in~\cite[Lemma~4.6]{fujiki1978closedness}, while the equivalence of (b) and (c) follows from the holomorphic Hironaka Theorem, see~\cite[Remark~4.4]{fujiki1978closedness}. Neither of these arguments works in the foliated case. We therefore chose the foliated version of (a) as the definition of Fujiki foliations. We expect that the Hodge decomposition also holds for foliations satisfying the foliated version of~(b). It would be also interesting to formulate a foliated version of~(c) and prove the Hodge decomposition for basic cohomology in this setting.
\end{rmk}

\section{Basic Dolbeault cohomology of the canonical foliations on complex moment-angle manifolds}

\subsection{Complex moment-angle manifolds and their canonical foliations.}\label{subsecfol}
Let $\mathcal K$ be a simplicial complex on the set $[m]=\{1,2,\ldots,m\}$, that is, $\mathcal K$ is a collection of subsets $ I \subset [m]$ such that if $I \in \mathcal{K}$ then each $J \subset I$ also belongs to $\mathcal K$. We assume that the empty set $\varnothing$ is in~$\mathcal K$. A one-element subset $\{i\}\subset[m]$ is a \emph{vertex} if $\{i\}\in\mathcal K$; otherwise it is a \emph{ghost vertex}.

Consider the $m$-dimensional unit polydisc
\[
  \mathbb{D}^m=\{(z_1,\ldots,z_m)\in \mathbb{C}^m\colon |z_i| \le 1\text{ for }i=1,\ldots,m \}
\]
and for each $I\subset[m]$ define the subspace
\[
  D_I=\{(z_1,\ldots,z_m)\in\mathbb{D}^m\colon |z_i|=1
  \text{ for $i\notin I$}\}=\prod_{i \in I} \mathbb{D} \times \prod_{i    \notin I} \mathbb{S},
\]
where $\mathbb S$ is the boundary of the unit disc $\mathbb D$.

The \emph{moment-angle complex} $\mathcal{Z}_{\mathcal{K}}$ corresponding to a simplicial complex $\mathcal{K}$ is defined as
\[
 \mathcal{Z}_{\mathcal{K}}=\bigcup\limits_{I \in \mathcal{K}}D_I\subset\mathbb{D}^m.
\] 
The moment-angle complex is equipped with the natural action of the torus 
\[
  T^m=\{(t_1,\ldots,t_m)\in\C^m\colon |t_i|=1\}.
\]
When $\SK$ is a simplicial subdivision of an $(n-1)$-dimensional sphere, $\momang$ is a topological manifold of dimension $m+n$, see~\cite[Theorem~4.1.4]{buchstaber2015toric}, called the \emph{moment-angle manifold}.

We set $\C^\times=\C\setminus\{0\}$ and for each $I\subset[m]$ define
\[
  U_I=\{(z_1,\ldots,z_m)\in\C^m\colon z_i\ne 0
  \text{ for $i\notin I$}\}=\prod_{i \in I} \mathbb{C} \times 
  \prod_{i\notin I} \mathbb{C}^\times.
\]
By analogy with $\momang$ we define an open submanifold 
\[
  U(\SK)=\bigcup\limits_{I \in \mathcal{K}}U_I\subset\C^m.
\] 
Alternatively, $U(\SK)$ can be defined as the complement of a coordinate subspace arrangement:
\[
  U(\SK)=\C^m\setminus\bigcup_{\{i_1,\ldots,i_k\}\notin\SK}
  \{z_{i_1}=\cdots=z_{i_k}=0\},
\]
see~\cite[Proposition~4.7.3]{buchstaber2015toric}.

The manifold $U(\SK)$ has a coordinate-wise action of the algebraic torus $(\mathbb{C}^{\times})^m$, in which $T^m$ is a maximal compact subgroup. Furthermore, 
$U(\SK)$ is a toric variety  with the corresponding fan given by
\begin{equation}\label{coordfan}
	\Sigma_\SK = \{ \R_\ge\langle e_i \colon i \in I\rangle \colon I \in \SK\}, 
\end{equation}
where $e_i$ denotes the $i$-th standard basis vector of $\R^m$ and $\R_\ge\langle A\rangle$ denotes the cone spanned by a set of vectors~$A$. 

\medskip

We consider moment-angle manifolds $\momang$ with a $T^m$-invariant complex structure. The necessary and sufficient conditions for the existence of such a structure were established in~\cite{panov2012complex} and~\cite{ishida2013complex}; a short account is given below.

Assume that $\dim\momang=m+n$ is even; this can always be achieved by adding ghost vertices to~$\SK$.
A $T^m$-invariant complex structure on $\momang$ is defined by two pieces of data:
\begin{itemize}
\item[--] a complete simplicial fan $\Sigma=\{\SK;a_1,\ldots,a_m\}$ in $\R^n$ with underlying simplicial complex $\SK$ and fixed generators $a_1,\ldots,a_m$ of one-dimensional cones (a \emph{marked fan});

\item[--] a choice of a complex structure in the kernel of the linear map
\begin{equation}\label{qmap}
  q\colon\R^m\to\R^n,\quad e_i\mapsto a_i.
\end{equation}
\end{itemize}

A choice of a complex structure in $\Ker q$ is equivalent to a choice of an $\frac{m-n}2$-dimensional complex subspace $\mathfrak h\subset\C^m$ satisfying the two conditions:
\begin{itemize}
	\item[(a)] the composite $\h\hookrightarrow\C^m\stackrel{\mathrm{Re}}\longrightarrow \R^m$ is injective; 
	\item[(b)] the composite
	$\h\hookrightarrow\C^m\stackrel{\mathrm{Re}}\longrightarrow \R^m
	\stackrel q\longrightarrow\R^n$ is zero. 
\end{itemize}
Consider the $\frac{m-n}2$-dimensional complex-analytic subgroup
\[
  H=\exp(\mathfrak h)\subset(\C^\times)^m.
\]
By \cite[Theorem~3.3]{panov2012complex}, the holomorphic action of $H$ on $U(\SK)$ is free and proper, and the complex manifold $U(\SK)/H$ is $T^m$-equivariantly homeomorphic to~$\momang$. This defines a complex-analytic structure on any even-dimensional moment-angle manifold $\momang$ such that $\SK$ is the underlying complex of a complete simplicial fan.

Conversely, assume that a moment-angle manifold $\momang$ admits a complex structure invariant under the action of~$T^m$. By~\cite[Theorem 7.9]{ishida2013complex}, the manifold $\momang$ is $T^m$-equivariantly biholomorphic to the quotient $U(\SK)/H$ as above. The marked fan $\Sigma$ and the complex subspace $\mathfrak h\subset\C^m$ are recovered as follows. The action of $T^m$ on $\momang$ extends to a holomorphic action of $(\C^\times)^m$ on~$\momang$, although the latter action is not effective. The global stabilisers subgroup (the noneffectivity kernel)
\[
  H= \{ g \in (\C^\times)^m \colon g\cdot x = x \text{ for all $x \in \momang$}\}
\]  
is a complex-analytic subgroup of $(\C^\times)^m$. The Lie algebra $\h$ of $H$ is a complex subalgebra of the Lie algebra $\C^m$ of $(\C^\times)^m$. 
By \cite[Proposition 7.8]{ishida2013complex}, it satisfies the following:
\begin{itemize}
	\item[(a)] the composite $\h\hookrightarrow\C^m\stackrel{\mathrm{Re}}\longrightarrow \R^m$ is injective; 
	\item[(b)] the quotient map $q \colon \R^m \to \R^m/\mathop{\mathrm{Re}}(\h)$ sends the fan $\Sigma_\SK$ to a complete fan $\Sigma=q(\Sigma_\SK)$ in $\R^m/\mathop{\mathrm{Re}}(\h)$. 
\end{itemize}
Here we identify $\R^m$ with the Lie algebra $\mathfrak{t}$ of~$T^m$.

\medskip

Now we proceed to describe the canonical holomorphic foliation on~$\momang$.
Define a real Lie subalgebra and the corresponding Lie group
\[
  \mathfrak r=\Ker q=\mathop{\mathrm{Re}}(\h) \subset \R^m = \mathfrak{t}, \qquad R=\exp(i\r) \subset T^m.
\]
The complexification $\mathfrak r^\C\subset\C^m$ is $\Ker q^\C$, where $q^\C\colon\C^m\to\C^n$ is the complexification of~\eqref{qmap}.
Now define the complex $(m-n)$-dimensional Lie group
\begin{equation*}
  R^{\mathbb{C}}=\exp(\r^\C)=\exp(\Ker q^\C)\subset(\C^\times)^m.
\end{equation*}
The restriction of the $(\C^\times)^m$-action on $U(\SK)$ to $R^\C$ has discrete stabilisers. We therefore obtain a holomorphic foliation of $U(\SK)$ by the orbits of $R^\C$. We denote this foliation by~$\F_{\Sigma}$. In more precise terms, the stabilisers of the  $R^\C$-action and the leaves of the foliation $\F_{\Sigma}$ can be identified as follows.

\begin{proposition}[{\cite[Proposition~4.2]{p-u-v16}}]\label{leaves}
For any subset $I\subset [m]$, define the coordinate subspace
$\C^I=\C\langle e_k\colon k\in I\rangle\subset\C^m$ and the
subgroup
\[
  \Gamma_I=\Ker q^\C\cap\bigl(\Z\langle2\pi i e_1,\ldots,2\pi i e_m\rangle+\C^I\bigr).
\]
\begin{itemize}
\item[(a)] $\Gamma_I$ is a discrete subgroup of $\C^m$ whenever $I\in\SK$.

\item[(b)] 
A leaf $R^\C z$ of the foliation $\mathcal F_\Sigma$ is biholomorphic to
\[
  \Ker q^\C/\Gamma_I\cong
  (\C^\times)^{\mathop{\mathrm{rk}}\Gamma_I}\times\C^{m-n-
  \mathop{\mathrm{rk}}\Gamma_I},
\]  
where $I\in\SK$ is the set of zero coordinates of $z\in U(\SK)$.
\end{itemize}
\end{proposition}

Note that the leaf through a generic point (with no vanishing coordinates) is $R^\C\cong\Ker q^\C/\Gamma_\varnothing$ in the notation above.

The holomorphic foliation $\mathcal F_\Sigma$ is mapped by the quotient projection $U(\SK)\to U(\SK)/H$ to a holomorphic foliation of $\zk\cong U(\SK)/H$ by the orbits of $R^\C/H\cong R$. We denote this latter foliation by $\F_\h$ and refer to it as the \emph{canonical holomorphic foliation} of~$\zk$.

\subsection{Stellar subdivisions}\label{subm}

Let $\momang=U(\SK)/H$ be a complex moment-angle manifold, where $\SK$ is the simplicial complex underlying a complete $n$-dimensional marked simplicial fan $\Sigma$ with $m$ generating vectors $a_1,\ldots,a_m$. 

\begin{constr}[rational stellar subdivision] 
Choose a $k$-dimensional cone ${\tau \in \Sigma}$, $k>1$. We may assume without loss of generality that $\tau$ is generated by $a_1,\ldots,a_k$, and denote the corresponding simplex by $I=\{1,\ldots,k\}\in \SK$. We construct a new complete marked simplicial fan $\Sigma_{\tau}$ with $m+1$ generators $a_0=\alpha_1a_1+\cdots+\alpha_ka_k$, $a_1,\ldots,a_m$, where $\alpha_i \in \mathbb{N}$ and the maximal cones are described as follows. If $\sigma$ is a maximal cone of $\Sigma$ not containing~$\tau$, then $\sigma$ is also a maximal cone of~$\Sigma_\tau$. If $\tau\subset\sigma$ and $\sigma$ is a maximal cone of~$\Sigma$, then we replace $\sigma$ by $k$ maximal cones $\sigma_1,\ldots,\sigma_k$, where $\sigma_i$ is obtained by replacing $a_i$ in the generator set of $\sigma$ by $a_0=\alpha_1a_1+\cdots+\alpha_ka_k$. The fan $\Sigma_{\tau}$ is called a \emph{rational stellar subdivision} of $\Sigma$ at~$\tau$. It depends on the choice of positive integer parameters $\alpha_i$, although we do not reflect this in the notation. The standard \emph{stellar subdivision} corresponds to $\alpha_1=\cdots=\alpha_k=1$.

The underlying simplicial complex of $\Sigma_\tau$ is the stellar subdivision of $\sK$ at $I$, which we denote by $\sK_\tau$ or $\st(I,\SK)$ (it does not depend on the~$\alpha_i$).
\end{constr}


\begin{constr}[generalised toric blow-up]\label{folbu}
Here we define a holomorphic foliated surjection $(U(\sK_\tau), \F_{\Sigma_{\tau}})\to (U(\sK),\F_\Sigma)$, which in the case of rational fans covers the blow-down map of toric varieties. Consider projection~\eqref{qmap} corresponding to the fan $\Sigma_\tau$:
\[
  q_\tau\colon\R^{m+1}\to\R^n,\quad e_0\mapsto \alpha_1a_1+\cdots+\alpha_ka_k,\; e_i\mapsto a_i,\; 
  i=1,\ldots,k.
\]
We have
\begin{equation}\label{kerqt}
  \Ker q_\tau=\langle \alpha_1e_1+\cdots+\alpha_ke_k-e_0,\Ker q\rangle,
\end{equation}
where $\Ker q\subset\C^m$ is viewed as a subspace of $\C^{m+1}$ via the inclusion~$\C^m\to\C^{m+1}$ on the last $m$ coordinates. 
Then $\F_{\Sigma_{\tau}}$ is the foliation of $U(\sK_\tau)$ by the orbits of $R^\C_\tau=\exp(\Ker q^\C_\tau)$. 

Now define a holomorphic surjective map
\[
  f_\tau\colon\C^{m+1}\to\C^m,\quad (z_0,z_1,\ldots,z_m)\mapsto
  (z_0^{\alpha_1}z_1,\ldots,z_0^{\alpha_k}z_k,z_{k+1},\ldots,z_m).
\]
and its exponential
\begin{align*}
  \varphi_\tau\colon(\C^\times)^{m+1}&\to(\C^\times)^m,\\
  (e^{w_0},e^{w_1},\ldots,e^{w_m})&\mapsto
  (e^{\alpha_1w_0+w_1},\ldots,e^{\alpha_kw_0+w_k},e^{w_{k+1}},\ldots,e^{w_m}).
\end{align*}

\begin{proposition}\label{folU}
The map $f_\tau$ restricts to a holomorphic foliated surjection
\[
  f_{\tau}\colon (U(\sK_\tau), \F_{\Sigma_{\tau}}) \rightarrow 
  (U(\SK), \F_{\Sigma}).
\]
\end{proposition}
\begin{proof}
Recall that $U(\sK)$ is the toric variety corresponding to the fan ${\Sigma_{\mathcal K}}$, see~\eqref{coordfan}, and $U(\sK_\tau)$ corresponds to $\Sigma_{\sK_\tau}$. One checks easily that the map $f_\tau\colon\C^{m+1}\to\C^m$ restricts to a toric morphism  $U(\sK_\tau)\to U(\SK)$, namely, the morphism induced by the map of fans 
$\Sigma_{\sK_\tau}\to\Sigma_{\mathcal K}$ sending $e_0$ to $\alpha_1e_1+\cdots+\alpha_ke_k$ and sending $e_i$ to $e_i$ for $i=1,\ldots,m$.


The map $f_\tau$ is $\varphi_\tau$-equivariant, that is, 
$f_\tau(g\cdot z)=\varphi_\tau(g)\cdot f_\tau(z)$ for $z\in\C^{m+1}$ and $g\in(\C^\times)^{m+1}$. Furthermore, 
$\varphi_\tau(R^\C_\tau)=R^\C$ by inspection. It follows that $f_\tau$ takes $R^\C_\tau$-orbits (leaves of~$\F_{\Sigma_{\tau}}$) to $R^\C$-orbits (leaves of~$\F_\Sigma$), and therefore defines a holomorphic foliated map.
\end{proof}

Let $\C^{m-k}\subset\C^m$ be the coordinate subspace with the first $k$ coordinates vanishing. The preimage of any $x=(0,\ldots,0,x_{k+1},\ldots,x_m)\in U(\sK)\cap\C^{m-k}$ under the map 
$f_{\tau}\colon U(\sK_\tau) \rightarrow U(\SK)$ is given by
\[
  f_\tau^{-1}(x)=\{0\}\times(\C^k\setminus\{0\})\times
  \{(x_{k+1},\ldots,x_m)\}\subset\C^{m+1}.
\]
On the other hand, the $f_\tau$-preimage of $y=(y_1,\ldots,y_m)\in U(\sK)\setminus\C^{m-k}$ is the orbit of $(1,y_1,\ldots,y_m)$ under the action of the one-parameter subgroup 
\begin{multline*}
  \exp\langle \alpha_1e_1+\cdots+\alpha_ke_k-e_0\rangle\\=\{(e^{-t},e^{\alpha_1t},  
  \ldots,e^{\alpha_kt},1,\ldots,1),\;t\in\R\}\subset(\C^\times)^{m+1}.
\end{multline*}
\end{constr}

We refer to the holomorphic surjection $f_{\tau}\colon (U(\sK_\tau), \F_{\Sigma_{\tau}}) \rightarrow (U(\SK), \F_{\Sigma})$ as a \emph{generalised toric blow-down map}.

\begin{rmk}
When $\Sigma$ is a rational fan, both $R^\C\subset(\C^\times)^m$ and $R_\tau^\C\subset 
(\C^\times)^{m+1}$ are closed subgroups. In the case $\alpha_1=\cdots=\alpha_k=1$, the map $f_\tau\colon U(\sK_\tau)\to U(\sK)$ covers the standard blow-down map $V_{\Sigma_\tau}\to V_\Sigma$ of the quotient toric varieties $V_{\Sigma_\tau}=U(\sK_\tau)/R^\C_\tau$ and $V_{\Sigma}=U(\sK)/R^\C$. 
\end{rmk}

\subsection{Subdividing to a polytopal fan.}
A fan $\Sigma$ is \emph{polytopal} if it is the normal fan of a (bounded) convex polytope. Equivalently, a fan is polytopal if it can be obtained
by taking cones with apex~$0$ over the faces of a convex polytope in~$\R^n$ containing $0$ in its interior. Note that a polytopal fan is complete.

As is well known in toric geometry, rational polytopal fans correspond to projective toric varieties. The \emph{toric Chow Lemma}~\cite{dani78} states that for any complete toric variety $V$, there is an equivariant surjective birational morphism $V'\to V$ from a nonsingular projective (or K\"{a}hler) toric variety~$V'$. Furthermore, a nonsingular projective $V'$ can be obtained from $V$ as the result of a sequence of blow-ups along torus-invariant subvarieties, see~\cite[Theorem 4.5]{abramovich1999note}. In the language of fans, for any complete rational fan $\Sigma$ there is a sequence of stellar subdivisions turning $\Sigma$ into a nonsingular polytopal fan. A combinatorial version of this result for 
PL spheres is proved in~\cite{adiprasito2015derived}. 

Here we prove a generalisation of this result to non-rational fans (Theorem~\ref{hir}) and foliated moment-angle manifolds (Theorem~\ref{hirmom}). The proof is by adapting the results of Abramovich--Matsuki--Rashid~\cite{abramovich1999note}, de~Concini--Procesi~\cite{concini1985completesymm} and Adiprasito--Izmestiev~\cite{adiprasito2015derived} to the non-rational case.

Let $\Sigma$ be a marked simplicial fan in $\R^n$ with generators $a_1, \ldots, a_m$. We refer to a marked fan $\Sigma'$ with generators $a_1, \ldots, a_m, b_1, \ldots, b_s$ as a \emph{rational subdivision} of~$\Sigma$ if $\Sigma'$ is a subdivision of~$\Sigma$ and for any cone $\sigma=\R\langle a_{i_1}, \ldots, a_{i_k}\rangle$ of $\Sigma$ and any $b_j \in \sigma$, the vector $b_j$ belongs to the lattice generated by $a_{i_1}, \ldots, a_{i_k}$.

\begin{lemma}\label{stellarsubdivisionofsubdision}
Let $\Sigma$ be a marked simplicial fan in~$\R^n$, and let $\Sigma'$ be its rational subdivision. Then there is a sequence of stellar subdivisions of $\Sigma$ at 2-dimensional cones such that resulting fan $\Sigma''$ is a subdivision of~$\Sigma'$.
\end{lemma}
\begin{proof}
In the case of a rational fan $\Sigma$ this is proved in~\cite[\S2.4, Proposition]{concini1985completesymm}. For arbitrary $\Sigma$, we can argue cone-wise: it is enough to show that for any maximal cone $\sigma \in \Sigma$ there is a sequence of stellar subdivisions of $\sigma$ such that the resulting union of cones is a subdivision of $\sigma\cap\Sigma'$. This follows by applying the rational result to the rational fan $\sigma$ and its rational subdivision~$\sigma\cap\Sigma'$. (Note that the fan formed by all faces of a simplicial cone $\sigma$ is rational with respect to the lattice spanned by any generator set of~$\sigma$.)
\end{proof}

\begin{lemma}\label{polytopality}
Let $\Sigma$, $\Sigma'$ and $\Sigma''$ be as in Lemma~\ref{stellarsubdivisionofsubdision}. Assume $\Sigma'$ is polytopal. Then $\Sigma''$ is also polytopal.
\end{lemma}
\begin{proof}
This is proved in~\cite[Claim~3]{adiprasito2015derived}. Namely, we first choose a function $h''$ that is linear on the $n$-dimensional cones of $\Sigma''$ and is strictly convex across all $(n-1)$-dimensional cones of $\Sigma''$ except $(n-1)$-dimensional cones of~$\Sigma$. Since $\Sigma'$ is a polytopal fan, there is also a function $h'$ that is linear on the cones of $\Sigma'$ and is strictly convex across every $(n-1)$-dimensional cone of~$\Sigma'$. Since $\Sigma''$ is a subdivision of $\Sigma'$, the function $h'$ is also linear on the cones of~$\Sigma''$ and it is strictly convex across those $(n-1)$-dimensional cones of $\Sigma''$ where the convexity of $h''$ can fail. Now, for an $\varepsilon>0$ small enough, $h'+\varepsilon h''$ is strictly convex at all $(n-1)$-dimensional cones of~$\Sigma''$.
\end{proof}

\begin{rmk}
In \cite[Theorem~4.5]{abramovich1999note}, the authors prove a rational version of Lemma~\ref{polytopality} using properties of projective and separated morphisms of toric varieties.
It would be interesting to define these concepts and establish their appropriate properties in the category of non-commutative toric varieties (or foliated moment-angle manifolds).
\end{rmk}

\begin{thm}\label{hir}
Let $\Sigma$ be a complete simplicial fan in $\R^n$ with generators $a_1,\ldots, a_m$. Then there exists a sequence of stellar subdivisions of $\Sigma$ such that the resulting fan $\Sigma'$ is polytopal.
\end{thm}
\begin{proof}
Consider the hyperplane arrangement formed by the hyperplanes containing all $(n-1)$-dimensional cones of~$\Sigma$. As observed in~\cite[Lemma~6.9.2]{dani78} (see also~\cite[Proposition 2.17]{oda88}), the resulting fan $\Sigma_h$ is polytopal and is a subdivision of~$\Sigma$. Let $l_1,\ldots, l_s$ be the new $1$-dimensional cones of $\Sigma_h$. Suppose $l_j$ lies in the interior of a cone $\sigma_j$ of~$\Sigma$. Close to $l_j$ choose a ray $l_j'$ which is rational with respect to the lattice spanned by the generators of~$\sigma_j$. The fan obtained from $\Sigma_h$ by replacing $l_j$ by~$l'_j$ is still polytopal, as the polytopality is an open condition for simplicial fans. Applying this consecutively to all $l_j$'s we obtain a rational subdivision $\Sigma_h'$ of $\Sigma$ which is polytopal.  

By Lemma~\ref{stellarsubdivisionofsubdision}, there is a sequence of stellar subdivisions of $\Sigma$ such that the resulting fan $\Sigma'$ is a subdivision of the polytopal fan~ $\Sigma'_h$. Hence, the fan $\Sigma'$ is polytopal by Lemma~\ref{polytopality}.
\end{proof}

Now consider a complex moment-angle manifold $\momang$ with the corresponding marked fan $\Sigma$ and the holomorphic foliation $\F_{\h}$ defined in Subsection~\ref{subsecfol}. By~\cite[Proposition~4.4]{p-u-v16}, the foliation $\F_\h$ is transversely K\"{a}hler if $\Sigma$ is a polytopal fan. 
We have the following foliated version of toric Chow Lemma.

\begin{thm}\label{hirmom}
For any complex moment-angle manifold $\momang$ with the corresponding marked fan $\Sigma$, there exists a holomorphic foliated surjection 
\[
f_{\mathcal{Z}} \colon (\mathcal{Z}_{\SK'}, \F_{\h'}) \longrightarrow (\momang, \F_{\h})
\]
where $(\mathcal{Z}_{\SK'}, \F_{\h'})$ is a transverse K\"{a}hler foliated moment-angle manifold corresponding to a polytopal marked fan $\Sigma'$, and  $f_{\mathcal{Z}}$ is induced by a foliated surjection 
\[
 f\colon (U(\SK'), \F_{\Sigma'}) \longrightarrow (U(\SK), \F_{\Sigma})
\]
given by a composition of generalised toric blow-ups.
\end{thm}

\begin{proof}
We apply Theorem~\ref{hir} to construct the marked polytopal fan~$\Sigma'$. We may assume that $(m'-n)$ is even, as otherwise we can apply one more stellar subdivision at a maximal cone of $\Sigma'$.  Hence, we obtain a foliated surjection
\[
 f\colon (U(\SK'), \F_{\Sigma'}) \longrightarrow (U(\SK), \F_{\Sigma})
\]
as in Proposition~\ref{folU}.
Let $q^{\mathbb{C}}\colon \mathbb{C}^m \rightarrow \mathbb{C}^n$ and $q'^{\mathbb{C}}\colon \mathbb{C}^{m'} \rightarrow \mathbb{C}^n$ be the linear maps corresponding to $\Sigma$ and $\Sigma'$ respectively. We may regard $\Ker q^\C$ as a subspace in $\C^{m'}$ via the natural inclusion $\C^m\subset\C^{m'}$. Performing a single stellar subdivision at a simplex $\{j_1,\ldots,j_a\}$ adds to $\Ker q^\C$ an one-dimensional subspace generated by $e_{j_0}-e_{j_1}-\cdots-e_{j_a}$, see~\eqref{kerqt}. We therefore have 
\[
  \Ker q'^{\mathbb{C}}=\langle V,\Ker q_{\mathbb{C}}\rangle,
\]
where $V$ is the subspace generated by the vectors $e_{j_0}-e_{j_1}-\cdots-e_{j_a}$ corresponding to all stellar subdivisions in the sequence.

We pick an arbitrary half-dimensional complex subspace $\h_0 \subset V$ such that its projection onto the real part $\mathbb{R}^{m'}$ is injective. Then the vector space $\h_0 \oplus \h$ provides a complex structure on~$\mathcal{Z}_{\SK'}$. It is clear from the description of generalised toric blow-up that $f$ maps any orbit of $\exp(V)$ on $U(\SK')$ to a single point. It follows that there is a holomorphic foliated map $f_{\mathcal{Z}}$ closing the commutative diagram
\[
\begin{tikzcd}
      (U(\SK'), \F_{\Sigma'}) \arrow{r}[swap]{f}\arrow{d}{\pi_{\Sigma'}} & (U(\SK), \F_{\Sigma})\arrow{d}{\pi_{\Sigma}}  \\
     (\mathcal{Z}_{\SK'}, \F_{\h_0 \oplus \h})\arrow{r}[swap]{f_{\mathcal{Z}}} & (\momang, \F_{\h})
\end{tikzcd}\qedhere
\]
\end{proof}

\subsection{Basic Dolbeault cohomology.}
\begin{thm}\label{mainhodge}
Let $\momang$ be a complex moment-angle manifold corresponding to a complete simplicial fan $\Sigma$, and let $\F_{\h}$ be its canonical foliation. Then  $(\momang,\F_{\h})$ is a Fujiki foliation, so there is a Hodge decomposition
\[
H_{\F_{\h}}^r(\momang;\mathbb{C})=\bigoplus_{p+q=r} H^{p,q}_{\F_{\h}}(\momang).
\]
Furthermore,
\[
    H^{p,q}_{\F_{\h}}(\momang)=0\quad \text{ if } p \neq q.
\]
\end{thm}
\begin{proof}
%
First assume that $\Sigma$ is polytopal. Then the canonical foliation $\F_\h$ is 
transversely K\"{a}hler by~\cite[Proposition~4.4]{p-u-v16}, so its basic cohomology admits a Hodge decomposition by~\cite[Theorem~3.4.6]{kacimi1990operateurs}. The fact that non-trivial basic Dolbeault cohomology groups appear only on the diagonal $p=q$ is shown in~\cite[Theorem 8.1]{ishida2018towards} in the transversely K\"ahler case. 

For a general fan $\Sigma$, we consider the map
\[
f_{\mathcal{Z}} \colon (\mathcal{Z}_{\SK'}, \F_{\h'}) \longrightarrow (\momang, \F_{\h})
\]
as in Theorem~\ref{hirmom}. Since the foliation $(\mathcal{Z}_{\SK'}, \F_{\h'})$ is transversely K\"{a}hler, the foliation $(\momang, \F_{\h})$ is Fujiki and both statements of the theorem follow from Lemma~\ref{Fujiki}. 
\end{proof}

Now the basic Dolbeault cohomology ring of $\momang$ can be described completely:

\begin{thm}\label{DolbeaultM}
There is an isomorphism of algebras:
\[
  H^{*,*}_{\F_{\h}}(\momang) \cong \mathbb{C}[v_1,\ldots,v_m] /(I_{\SK}+J),\quad
  v_i \in H^{1,1}_{\F_{\h}}(\momang),
\]
where $I_\SK$ is the Stanley--Reisner ideal of~$\SK$, generated by the monomials 
\[
  v_{i_1}\cdots v_{i_k}\quad\text{with }\{i_1,\ldots,i_k\}\notin\SK,
\]  
and $J$ is the ideal generated by the linear forms
\[
  \sum_{i=1}^m \langle u, a_i \rangle v_i\quad
  \text{with }u\in (\mathfrak t/\r)^*.
\]
\end{thm}
\begin{proof}
This follows from the description of the basic de Rham cohomology of $\zk$~\cite[Theorem 3.4]{ikp2018basic} and the Hodge decomposition of Theorem~\ref{mainhodge}.
\end{proof}

\section{Manifolds with maximal torus actions.}\label{max}
Here we generalize results of the previous section to the class of \emph{complex manifolds with maximal torus action}, which includes complex moment-angle manifolds as well as LVM- and LMVB-manifolds.

We briefly recall the notation, referrring to~\cite{ishida2013complex}, \cite[Section 4]{ikp2018basic} for the details. Denote by $\Max_1$ the category of complex manifolds with maximal torus action. The objects of $\Max_1$ are given by triples $(M,G,y)$, where

\begin{itemize}
  \item $M$ is a compact connected complex manifold;
  \item $G$ is a compact torus acting on $M$, the $G$-action preserves the complex structure on $M$ and is  \emph{maximal}, i.\,e. there exists a point $x \in M$ such that $\dim G+\dim G_x = \dim M$; 
  \item $y \in M$ satisfies $G_y = \{1\}$. 
\end{itemize}

The morphisms $\Hom_{\Max_1} ((M_1,G_1,y_1), (M_2,G_2,y_2))$ are defined as pairs $(f, \alpha)$, where $\alpha\colon G_1 \rightarrow G_2 $ is a smooth homomorphism and $f \colon M_1 \rightarrow M_2$ is an $\alpha$-equivariant holomorphic map such that $f(y_1)=y_2$.

As in the case of moment-angle manifolds, we may associate with $(M,G,y) \in \Max_1$ a unique pair $(\Sigma', \h)$ such that $M$ is $G$-equivariantly biholomorphic to the quotient manifold $V_{\Sigma'}/H$. Here $\Sigma'$ is a nonsingular rational fan in the Lie algebra $\g$ of~$G$ and $\h \subset \g^\C$ is a complex Lie subalgebra,  $V_{\Sigma'}$ is the toric variety corresponding to $\Sigma'$, and $H$ is the subgroup of the algebraic torus $G^\C$ corresponding to $\h \subset \g^\C$. 
 
There is also a \emph{canonical holomorphic foliation} $\F_\h$ on a manifold $(M, G, y)$ which is defined by the action of the group $R=\exp(i\r)$, where $\r=\mathop{\mathrm{Re}}(\h)$. 
 
The image of the fan $\Sigma'$ under the projection $q\colon \g\rightarrow \g /\r$  is a complete fan. We let $\Sigma=q(\Sigma')$ and denote the generators of $\Sigma$ by $a_1,\ldots,a_m$.

Now we are ready to generalise Theorem~\ref{DolbeaultM} to all complex manifolds with maximal torus action:

\begin{thm}\label{dolbgene}
The basic Dolbeault cohomology ring of the canonical foliation on
a complex manifold with maximal torus action $(M,G,y)$ is given by
\[
  H^{*,*}_{\F_\h}(M) \cong \mathbb{C}[v_1,\ldots,v_m] /(I_{\SK}+J),\quad
  v_i \in H^{1,1}_{\F_\h}(M),
\]
where $I_\SK$ is the Stanley--Reisner ideal of~$\SK$
and $J$ is the ideal generated by the linear forms
\[
  \sum_{i=1}^m \langle u, a_i \rangle v_i\quad
  \text{with }u\in (\g/\r)^*.
\]
\end{thm}
\begin{proof}
It is proved in \cite[Section 5]{ikp2018basic} that $(M, \F_\h)$ is transversely equivalent via morphisms in $\Max_1$ to the canonical foliation on a moment-angle manifold $\zk$ with the same $(\Sigma, \h)$-data (see \cite[Theorem~5.4, Theorem~5.7, Theorem~5.8]{ikp2018basic}). This transverse equivalence induces an isomorphism in basic Dolbeault cohomology, since morphisms in $\Max_1$ are holomorphic (see~\cite[Proposition~5.1]{ikp2018basic}). Now the result follows from Theorem~\ref{DolbeaultM}.
\end{proof}

\begin{rmk}
Construction~\ref{folbu} can be interpreted as a foliated blow-up of manifolds with maximal torus actions in the following way. Let $M=V_{\Sigma'}/H$ be such a manifold defined by a pair $(\Sigma',\mathfrak h)$, where $\Sigma'$ is a fan with underlying simplicial complex~$\sK$. Choose a cone $\tau$ of $\Sigma'$ and consider the stellar subdivisions 
$\Sigma'_\tau$ and $\sK_\tau$. Let $M_\tau=V_{\Sigma'_{\tau}}/H$ be the manifold defined by the pair $(\Sigma'_\tau,\mathfrak h)$ (with the same~$\mathfrak h$). We have a commutative diagram
\[
\begin{tikzcd}
     U(\SK_\tau) \arrow{r}\arrow{d} & U(\SK)\arrow{d} \\
     V_{\Sigma'_{\tau}}\arrow{r}\arrow{d} & V_{\Sigma'} \arrow{d} \\
     V_{\Sigma'_{\tau}}/H \arrow{r} & V_{\Sigma'}/H
\end{tikzcd}
\]
Here the vertical arrows are the quotient projections, $U(\SK_\tau)\to U(\SK)$ is the foliated surjection from Proposition~\ref{folU}, $V_{\Sigma'_{\tau}}\to V_{\Sigma'}$ is the blow-down map of toric varieties, and $V_{\Sigma'_{\tau}}/H\to V_{\Sigma'}/H$ can be regarded as a foliated blow-down map of manifolds with maximal torus action. Note that $\dim_\C U(\SK_\tau)=m+1$,  $\dim_\C U(\SK)=m$, while $V_{\Sigma'_{\tau}}$ and $V_{\Sigma'}$ (and therefore $V_{\Sigma'_{\tau}}/H$ and $V_{\Sigma'}/H$) have the same dimension.
\end{rmk}

\section{Dolbeault cohomology of moment-angle manifolds}\label{dcmam}
A DGA model for the Dolbeault cohomology algebra of a complex moment-angle manifold can be obtained by combining the description of the basic Dolbeault cohomology (Theorem~\ref{DolbeaultM}) with the recent results by Ishida and Kasuya~\cite{ishida2019transverse}. This model is similar to the one obtained in~\cite[Theorem 5.4]{panov2012complex} in the case of a rational fan, although the differential here is described in less explicit terms.

\begin{thm}\label{dgamodel}
There exists a complex $\frac{m-n}{2}$-dimensional $R$-invariant subspace $W\subset \Omega^ 1(\momang)^{R}$ such that there is a quasi-isomorphism
\begin{equation*}
\bigl(H_{\F_{\h}}^{*,*}(\momang) \otimes \Lambda (W^{1,0} \oplus
W^{0,1}), d_{\momang} \bigr) 
\stackrel\simeq\longrightarrow 
\bigl( \Omega^{*, *}(\momang), \overline{\partial} \bigr),
\end{equation*}
where $W \otimes \mathbb{C} \cong W^{1,0} \oplus
W^{0,1}$ and $d_{\momang}(W^{0,1})=0$, $d_{\momang}(W^{1,0}) \in H_{\F_{\h}}^{1,1}(\momang)$.
\end{thm}
\begin{proof}
By  \cite[Lemma 4.4, Proposition 4.9]{ishida2019transverse}, there exists an
$\h$-valued 1-form $\omega$ on $\momang$ satisfying
\begin{itemize}
  \item[(a)] $\iota_{X_v}\omega=v$ for all $v \in \h$;
  \item[(b)] $\omega$ is $T^m$-invariant.
\end{itemize}
The choice of $\omega$ is equivalent to the choice of a $T^m$-invariant holomorphic distribution $\mathcal{W}$ on $\momang$ which is complementary to $\F_{\h}$ at each point.

Let $J$ be the complex structure on $\momang$, and $J_{\h}$ the complex structure on~$\h$. Choose a real basis $u_1,\ldots,u_{\frac{m-n}{2}}, J_{\h}u_1,\ldots,J_{\h}u_{\frac{m-n}{2}}$ of $\h$ and write
\[
  \omega=\omega_1 \otimes u_1+\cdots+\omega_{\frac{m-n}{2}} \otimes u_{\frac{m-n}{2}}+\omega_1'\otimes J_{\h}u_1+\cdots+\omega_{\frac{m-n}{2}}' \otimes J_{\h}u_{\frac{m-n}{2}}.
\]
Then define $W$ as the $J$-invariant subspace $\langle \omega_1,\ldots,\omega_{\frac{m-n}{2}}, \omega_1',\ldots\omega_{\frac{m-n}{2}}' \rangle$.

Since $\F_{\h}$ is a Fujiki foliation, the $\partial_{\F_{\h}}\overline{\partial}_{\F_{\h}}$-lemma holds for the complex $\bigl(\Omega^{*,*}_{\F_{\h}}(\momang), \overline{\partial}_{\F_{\h}}\bigr)$, see Lemma~\ref{Fujiki}. Then the required quasi-isomorphism follows from \cite[Corollary 4.10, Proposition 4.11]{ishida2019transverse}.
\end{proof}

\end{document}